\documentclass[12pt,reqno]{amsart}
\usepackage{amsfonts}
 \usepackage{amssymb}
 \usepackage{verbatim}
 \title[Corona Theorem with Estimates]{Estimates for the Corona Theorem  on
$H^{\infty}_{\mathbb{I}}(\D)$}
\author{Debendra P.  Banjade}

\address{Department of Mathematics and Statistics\\
         Coastal Carolina University \\
         P.O. Box 261954\\
Conway, SC 29528-6054 \\
         (843) 349-6569}
\email{dpbandjade@coastal.edu}


 \newtheorem{thm}{Theorem}[section]
\newtheorem*{thm*}{Theorem}  
\newtheorem{lemma}{Lemma}[section]

\newtheorem*{Uchiyama}{Corona Theorem}
\newtheorem*{wolffthm}{Wolff's Theorem}

\newtheorem*{tr}{ Ideal Theorem (Treil)}
\newtheorem*{cr1}{ Corollary 1}
\newcommand{\ran}{\operatorname{ran}}    \newcommand{\C}{\mathbb{C}}
  
\newcommand{\I}{\mathcal{I}} \newcommand{\D}{\mathbb{D}}

\allowdisplaybreaks[1]

\begin{document}

\maketitle

\begin{abstract} Let $\mathbb{I}$ be a proper ideal of $H^{\infty}(\D)$. We prove the corona theorem for infinitely many generators in the algebra $H^{\infty}_{\mathbb{I}}(\D)$. This extends the finite corona results of Mortini, Sasane, and Wick  \cite{MSW}.  We also provide the estimates for corona solutions.  Moreover, we prove a generalized Wolff's Ideal Theorem  for this sub-algebra.  \end{abstract}

\section{Introduction}


 
Let $\mathbb{D}:=\{ z \in \mathbb{C}: |z|<1\}$ be an open unit disk in the complex plane $\C$ and $H^{\infty}(\D)$ be the set of all bounded analytic functions with the norm $ \Vert f \Vert_ {\infty} =\underset{z \in \D} \sup \vert f(z) \vert <\infty$.  In 1962, Carleson proved his famous corona theorem which states that the ideal, $ \I, $ generated by a finite set of functions $
\lbrace f_i \rbrace_{i=1}^{n} \subset H^{\infty}(\D) $ is the entire space $ H^{\infty}(\D),$  if for some $\epsilon > 0,$  $   \sum_{i=1}^{n} |f_i(z)|^2  \geq  \epsilon  \text{ for all }   z \in \D.$ 
  In 1979, Wolff gave a simplified proof of Carleson's corona theorem, which can be found in \cite{garnett}, that made use of $H^2$-Carleson's measures and Littlewood-Paley expressions. Both Carleson and Wolff provided the bounds for corona solutions depending on the number of functions $n$.  Later, Rosenblum \cite{rosenblum}, Tolokonnokov \cite{tolokonnikov} , and Uchiyama \cite{u}, independently, extended the corona theorem for infinitely many functions, where as the best estimate for the corona solution was due to Uchiyama as follows: 
\begin{Uchiyama}
Let  $ \lbrace f_i \rbrace_{i=1}^{\infty} \subset  H^{\infty}(\D)$, with
 $$0< \epsilon ^2 \leq \sum_{i=1}^{\infty} |f_i(z)|^2  \leq 1 \text{ for all } \; \;  z \in \D.$$
 Then there exist  $ \lbrace g_i \rbrace_{i=1}^{\infty} \subset  H^{\infty}(\D)$ such that 
 $$\sum_{i=1}^{\infty}  {f_i(z) g_i(z)}=1\; \text{ for all } \; \;  z \in \D$$ and 
 $$ \underset{z \in \D}{\sup} \{ \sum_{i=1}^{\infty} |g_i(z)|^2 \}  \leq \frac{ 9} {\epsilon^2}  \ln \frac{1}{\epsilon ^2} , \;  \text { for} \;  \epsilon ^2 < \frac{1}{e}. $$
 \end{Uchiyama}

The main purpose of this paper is to extend the corona theorem for infinitely many functions in  $ H^{\infty}_{\mathbb{I}}(\mathbb{D}).$  Moreover, we provide the estimates for the corona solutions.  This will completely settle the conjecture of  Ryle \cite{ryle1}.\\
 
The algebra, $ H^{\infty}_{\mathbb{I}}(\mathbb{D}),$ of our  interest is defined as follows:\\
Let $\mathbb{I}$ be any proper closed ideal in $ H^{\infty}(\D)$, and define
$$H^{\infty}_{\mathbb{I}}(\mathbb{D}):=\{ c+ \phi \; \vert \;  c\in \mathbb{C} \; \text{and} \; \phi \in \mathbb{I}\}. $$
Then  $H^{\infty}_{\mathbb{I}}(\mathbb{D})$ is a sub-algebra of $ H^{\infty}(\D)$. We regard $ \left ( H^{\infty}_{\mathbb{I}}(\mathbb{D})\right)_{l{^2}}$ as a sub-algebra of $ H_{l^2}^{\infty}(\mathbb{D})$, where $H_{l^2}^{\infty}(\mathbb{D})$ is a sequence of bounded analytic functions. Also, for $F = (f_1, f_2, . . . ), f_j \in  H^{\infty}(\mathbb{D}), $ we use the norm $$\Vert F \Vert _{\infty}  = \underset {z\in \D} \sup \left( \sum_{i=1}^{\infty} |f_i(z)|^2 \right)^{1/2}.$$

In  \cite{MSW},  Mortini, Sasane, and Wick proved the corona theorem for finitely many generators in $H^{\infty}_{\mathbb{I}}(\mathbb{D})$. In fact, \cite{MSW}  provided the estimates  on the solutions $g_j$ in terms of the parameters $\epsilon$ and n (the number of functions $f_j$). In this paper, we  prove an analogous result of Uchiyama  for the sub-algebra  $H^{\infty}_{\mathbb{I}}(\mathbb{D})$ by removing the dependency of estimates  on n.  \\


Let $ f \in H^{\infty}_{\mathbb{I}}(\mathbb{D})$, say $f(z)=c+ \phi (z)$, for $\phi \in \mathbb{I}$ and $c \in \C. $ For simplicity, we use the notation: 
$f(z)=f_c+\phi_f(z) ,$ where $f_c  \in \mathbb{C}$ and $\phi_f \in \mathbb{I}.$  Similarly, let $F = (f_1, f_2, . . . ), f_j \in  H^{\infty}_{\mathbb{I}}(\mathbb{D})$. Then for $z \in \D,$ we write $F(z)=F_c+\phi_F(z).$\\



We are now ready to state our Main Theorem, which extends to the corona theorem for infinitely many functions in $H^{\infty}_{\mathbb{I}}(\mathbb{D}).$
\begin{thm} \label{th1}
Let $ F(z) = (f_1(z), f_2(z), \dots),  \; f_j \in H^{\infty}_{\mathbb{I}}(\mathbb{D})$ and $$ 0< \epsilon ^2 \leq  F(z)F(z)^{\ast}\;  \leq 1 \; \text{ for all } \; z\in \D .$$ Then there exists $ U = (u_1(z), u_2(z), \dots ), \; u_{j} \in H^{\infty}_{\mathbb{I}}(\mathbb{D})$ such that $$ (a) \; F(z)U(z)^T = 1 \; \text{ for all}\;  z \in \D \; \; \text{and} $$
$$ (b) \; \Vert U \Vert _{\infty} \leq   \left( { 1+ \frac{1}{\Vert F_c \Vert }} \right) \frac{9}{\epsilon ^2} \ln \left( \frac{1}{\epsilon ^2}\right). $$ 

\end{thm}

In order to generalize the corona theorem, it is natural to ask if the corona theorem still holds true if we replace the lower bound, $\epsilon,$  in the corona condition  by any $H^{\infty}(\D)$ functions. Namely, let $h, f_1, f_2,...,f_n \in H^{\infty}(\D)$ such that 
\begin{equation} \label{data}
\vert h(z) \vert  \leq \sum_{i=1}^{n}  \vert f_i(z) \vert   \leq 1 \text{ for all } \; \;  z \in \D. 
\end{equation}
Then the question is does  (\ref{data})  always implies $h \in \mathcal {I} (f_1, f_2, ...,f_n),$ ideal generated by $ f_1, f_2, ...,f_n$?  Of course,  (\ref{data})  is a necessary condition, but the counter example provided by Rao \cite{rao} suggests that it is far from being sufficient. \\

\textit{Rao's Counter Example}: If $B_1$ and $B_2$ are  Blaschke products without common zeros for which
$\underset{z \in \D } \inf \left ( \vert B_1 (z) \vert + \vert B_2 (z) \vert \right)=0,$  then $ \vert B_1 B_2 \vert \leq \left ( \vert B_1 \vert ^2 + \vert B_2 \vert ^2 \right),$  but $B_1 B_2 \notin  \mathcal I \left ( B_1 ^2, B_2 ^2 \right ). $ \\

However,  T. Wolff's beautiful proof  (see \cite{garnett}, Theorem 2.3  in page 319)  showed that the condition (\ref{data})   is sufficient for $h^3 \in \mathcal {I} (f_1, f_2, ...,f_n).$  Wolff's Theorem can be rephrased as follows: 

\begin{wolffthm}

Let $ F(z) = (f_1(z), f_2(z), \dots, f_n(z)),  f_j \in H^{\infty}(\mathbb{D}), \;   h \in H^{\infty}(\mathbb{D}).$  If 

$$
 |h(z)| \leq \sqrt {F(z)F(z)^*}  \; \text{ for all } \; \;\;  z \in \mathbb{D},$$

then
$$ h^3 \in \mathcal{I} ( \{ f_j\}_{j=1}^n).$$

\end{wolffthm}
 
  But,  it was shown by Treil \cite{treil2} that this is not sufficient for $p=2$. \\
 
Many authors, independently, have considered this question, including  Cegrell \cite{cegrell1}, Pau \cite{pau}, Trent \cite{estimate}, and Treil \cite{T3}, for $p=1$. We refer this as a  problem of \lq  \lq ideal membership." It 
is Treil who has given the best known sufficient condition for ideal membership. We state Treil's Theorem as follows:

\begin{tr}
Let $ F(z)= (f_1(z), f_2(z),...),  \; f_j  \in H^{\infty}(\mathbb{D})$, $F(z) F(z)^{\ast} \leq 1  \text{ for all } \; z\in \D$,  and $ h\in H^{\infty}(\mathbb{D})$ such that 
$$  F(z)F(z)^{\ast}\;  \psi \left( F(z)F(z)^{\ast}\right) \geq |h(z)|  \text{ for all }  z\in \mathbb{D},$$ where $\psi:[0,1]\rightarrow [0,1]$ is a non-decreasing  function such that $\int _{0}^{1} \frac{\psi(t)}{t} dt <\infty.$ Then there exists $G \in H_{l^2}^{\infty}(\D)$ such that $$F(z)G(z)^{T}= h(z),\;\; \text{for all} \;\; z\in \D \text{.}$$
\end{tr}
An example of a function $ \psi $ that works in the case when $ F(z) $ is an $n$-tuple, $ n < \infty,  $ is 
\begin{equation*}
\psi(t) = \frac{1} {(\ln t^{-2})(\ln_2 t^{-2}) \dots (\ln_n t^{-2})(\ln_{n+1} t^{-2})^{1+\epsilon}} \text{,}
\end{equation*}
  where $\ln_k(t)=\underbrace {\ln \ln ... ln}_\text{ k+1 times}(t)$  and $ \epsilon > 0.$ \\

  %
  
  Applying Treil's  result,  we   extend the analogue of \lq \lq ideal theorem" on   $H^{\infty}_{\mathbb{I}}(\mathbb{D}).$ Recall that $  H^{\infty}_{\mathbb{I}}(\mathbb{D})$  is a sub-algebra  of $H^{\infty}(\D).$ Also, for $F = (f_1, f_2, . . . ), f_j =f_{c_j}+\phi _{f_j} \in  H^{\infty}_{\mathbb{I}}(\mathbb{D})$, we denote $F=F_c+\phi_F$. In the case  that  $F_c =0,$  several authors have given sufficient conditions   for ideal membership, for example, see \cite{GMN}, \cite{mortini},  and \cite{von}. For the case $F_c \neq 0$, we provide the following theorem:
  
\begin{thm} \label{th2}
Let $ F(z) = (f_1(z), f_2(z), \dots),  \; f_j \in H^{\infty}_{\mathbb{I}}(\mathbb{D})$ such that $F_c \neq 0$, and suppose  $$ \vert h(z) \vert  \leq  F(z)F(z)^{\ast} \psi \left( F(z)F(z)^{\ast} \right)\;  \leq 1 \; \text{ for all } \; z\in \D,$$ where $\psi$ is the function given in Treil's theorem.  Then there exists $ V = (v_1(z), v_2(z), \dots ), \; v_{j} \in H^{\infty}_{\mathbb{I}}(\mathbb{D})$ such that $$ (a) \; F(z)V(z)^T = h (z) \; \text{ for all}\;  z \in \D \; \; \text{and} $$
$$ (b) \; \Vert V \Vert _{\infty} \leq  C_0 \left( { 1+ \frac{1}{\Vert F_c \Vert}}  \right) , $$ where $C_0$ is the estimate for the $H^{\infty}(\D) $ solution obtained in \cite{T3}. 
\end{thm}


\begin{cr1}\label{cr}
Let $ F(z) = (f_1(z), f_2(z), \dots),  \; f_j \in H^{\infty}_{\mathbb{I}}(\mathbb{D})$ such that $F_c \neq 0$, and suppose  $$ \vert h(z) \vert  \leq  \sqrt{ F(z)F(z)^{\ast}  }  \leq 1 \; \text{ for all } \; z\in \D .$$ Then there exists $ V = (v_1(z), v_2(z), \dots ), \; v_{j} \in H^{\infty}_{\mathbb{I}}(\mathbb{D})$ such that $$ (a) \; F(z)V(z)^T = h^{3}  (z) \; \text{ for all}\;  z \in \D \; \; \text{and} $$
$$ (b) \; \Vert V \Vert _{\infty} \leq  C_1 \left( { 1+ {\frac{1}{\Vert F_c \Vert}}} \right) , $$ where $C_1$ is the estimate for the $H^{\infty}(\D) $ solution obtained in \cite{estimate}. 
\end{cr1}




\section{Preliminaries}

In this section, we discuss the method of our proofs and also provide some required lemmas.   To prove Theorem \ref{th1} and Theorem  \ref{th2}  in $H_{\mathbb{I}}^{\infty}(\D),$  we first find the corresponding solutions in the bigger algebra $H^{\infty}(\D).$  Then we add some correction terms on the $H^{\infty}(\D)$ - solutions to get the required solutions in our smaller algebra $H_{\mathbb{I}}^{\infty}(\D)$. For example, provided the corona condition, using Uchiyama version of corona theorem, we can easily find a solution $G$ in $ (H^{\infty}(\D))_{l^2}$ such that $F(z)G(z)^T=1 \; \text{ for all} \;  z \in \D$. But, our goal is finding a solution $U \in   \left (H^{\infty}_{\mathbb{I}}(\mathbb{D})\right)_{l^2}$ such that $F(z)U(z)^T=1$ for all $z\in \D.$ For this, if we can find an operator $Q$ so  that  
$M_Q (H^{\infty}(\D))_{l^2} \subseteq (H^{\infty}(\D))_{l^2}$ and for all $z \in \D$,  $\ran Q (z) = \ker F (z),$  then we can  construct the required solution  $U$ as $$U^{T}: = G ^T+Q X^T, $$ with a right choice of $X \in (H^{\infty}(\D))_{l^2} $. This solves our problem as follows:
$$F(z)U(z)^T=F(z)G(z)^T=1, \; \text{for all}\; z\in \D,$$ and the proper choice of $X$ will  make $U \in  \left (H^{\infty}_{\mathbb{I}}(\mathbb{D})\right)_{l^2}.$\\

The next lemma is a linear algebra result which gives us the desired $Q$ operator and so enables us to write down the most general pointwise solution of $F(z)U(z)^T=1.$ This lemma can be found in Ryle -Trent \cite{RT}, but we provide a proof for convenience.

\begin{lemma} \label{lemma}
Let  $\{ a_j \} _{j=1}^{\infty} \in l^2 $ and $A = ( a_1 , a_2 , \dots ) \in \mathcal{B} (l^2, \mathbb{C}).$ Then there exists a matrix $Q_A$ of order  $ \infty \times \infty$ such that the entries of $Q_A$ are either $\underset{-}+ a_j$ or $0$ and $Q_A$ satisfies:
\begin{equation} \label{ker}
\ran Q_A= \ker A   
\end{equation} 
 and $$(AA^{\ast})I _{l^2}- A^{\ast}A= Q_A Q_A^{\ast} \;  \; \text{with} \; \;  \Vert Q_A\Vert _{\mathcal{B}(l^2) } \leq \Vert A\Vert _{l^2}.$$  
 
Also, if  $\{d_j\}_{j=1}^{\infty} \in l^2$ and $D=(d_1 , d_2 , \dots )$, then 
 \begin{equation} \label{relation}
 (AD^T)I_{l^2}-D^TA=Q_A Q_D^T.
 \end{equation}
\end{lemma}

Following few examples should be helpful to understand the  Lemma \ref{lemma} in a simple way.\\

Let $f_1, f_2, ...,f_n \in H^{\infty}(\D)$ and fix $z\in \D$. Take $F= [f_1 \; \; f_2\; , ..., \; f_n]. $\\
  For $n=2,$
$F= [f_1 \; \; f_2],$ 
$Q_F=$ 
$\begin{bmatrix}
f_2\\
-f_1
\end{bmatrix}. $

Thus, 
$(FF^{\ast})I_2 - F^{\ast}F=  \begin{bmatrix}
\vert f_2 \vert ^2  & - \bar{f_1} f_2\\
\bar{f_2} f_1 &  \vert f_1\vert^2
\end{bmatrix} = Q_F Q_F^{\ast}.$

\vspace{12pt}
Also, for any  $D= \begin{bmatrix} 
d_1 & d_2
\end{bmatrix}$,  \\

$(FD^T)I_2-D^TF= \begin{bmatrix}
f_2  d_2  & - d_1 f_2\\
-d_2 f_1 &   f_1d_1
\end{bmatrix} = Q_F Q_D^T.$\\

\vspace{12pt}

Similarly, for  $n=3,$ we take  $F= \begin{bmatrix} 
f_1 & f_2 & f_3
\end{bmatrix}.$\\

So,

$Q_F= \begin{bmatrix}
f_2 & f_3 & 0\\
-f_1 & 0 & f_3\\
0 & -f_1 & -f_2
\end{bmatrix}. $\\

\vspace{12pt}
And, for  $n=4,$  $F= \begin{bmatrix} 
f_1 & f_2 & f_3 & f_4
\end{bmatrix}$ and \\

$Q_F= \begin{bmatrix}
f_2 & f_3 &  f_4 & 0 &0 & 0\\
-f_1 & 0  & 0 & f_3 & f_4 & 0\\
0 & -f_1 & 0&  -f_2 & 0 & f_4\\
0 & 0 & -f_1& 0 & -f_2 & f_3
\end{bmatrix}.$\\
\vspace{12pt}

Form the above pattern, it is easy to see that the operators $Q_F$'s  can be constructed inductively. Also, it is clear from (3),  applied to $A = F(z)$ and $Q_D = Q_{F(z)}$, that $\ran Q_F(z)= \ker F(z).$ \\

We are now ready to prove Lemma \ref {lemma}.\\

\begin{proof}[Proof of Lemma \ref{lemma}]
For $k \in \mathbb{N},$ define \\

$A_k = { \begin{bmatrix}
0  &  0 &  0 &  \hdots \\
\vdots & \vdots & \vdots & \ddots\\
c_{k+1} & c_{k+2} & c_{k+3}  & \hdots \\
-c_k & 0 & 0  & \hdots  \\
0 &  - c_k & 0 & \hdots  \\
0 &  0  & -c_k& \hdots  \\
\vdots & \vdots & \vdots & \ddots
\end{bmatrix}}$\\
\vspace{12pt}

Multiplying $A_k$  by $A_{k}^{\ast}$, we get \\

$A_k  A_{k}^{\ast} =  { \begin{bmatrix}
0  & \hdots &  0 &  0 & 0 & 0&  \hdots \\
\vdots & 0 & \vdots & \vdots & \vdots & \vdots & \hdots \\
0 & \hdots & 0 & 0 & 0 &  0 & \hdots \\
0 & \hdots  &  0 & \sum_{j=k+1}^{\infty} \vert c_j \vert ^2 & - \bar{c} _k c_{k+2} & - \bar{c}_k  c_{k+3}  & \hdots \\
0 & \hdots  &  0 & -c_k \bar{c}_{k+2} &  \vert c_k \vert ^2 & 0  & \hdots\\
0 & \hdots  &  0 & -c_k \bar{c}_{k+3} &  0 & \vert c_k \vert ^2  & \hdots \\
\vdots & \vdots & \vdots &  \vdots & \vdots & \vdots & \ddots
\end{bmatrix}}$\\

Hence,  \\
 $$\sum_{k=1}^{\infty} A_k  A_{k}^{\ast} =
 \begin{bmatrix}
\sum_{k \neq 1}^{\infty}  \vert c_k \vert ^2  & - \bar{c}_1  c_2   &  - \bar{c}_1  c_3 &   \hdots \\
- \bar{c}_2  c_1  & \sum_{k \neq 2}^{\infty}  \vert c_k \vert ^2 & - \bar{c}_2  c_3 & \hdots  \\
 - \bar{c}_3 c_1  & - \bar{c}_3 c_2 & \sum_{k \neq 3}^{\infty}  \vert c_k \vert ^2  &  \hdots \\
\vdots & \vdots & \vdots &  \ddots
\end{bmatrix}=CC^{\ast}I_{l^2}-C^{\ast}C.$$

Thus the required operator $Q_A$  can be defined   as 
$$ Q_A=[A_1,  A_2,....] \in  \mathcal{B} (\oplus_{1}^{\infty} l^2, l^2).$$

We note that (\ref{relation}) follows in a similar manner.\\

\end{proof}

We also need the following key lemma.

\begin{lemma} Assume that $\{{f_j}\}_{j=1}^{\infty} \subset H^{\infty}_{\mathbb{I}}(\mathbb{D})$ and 
$$ 0 < \epsilon^2 \leq \sum_{j=1}^{\infty} \vert f_j (z) \vert ^2 \leq 1\; \text{ for all} \; z \in \D.$$
Then   $$ \text {(a)} \;  \;    \epsilon ^2 \leq F_c F_c^{\star} = \sum_{j=1}^{\infty} \vert f_{c_j}  \vert ^2 \leq 1$$

and   $$\text{ (b)} \;  \;   \Vert \phi_F \Vert _{\infty} = \sup_{z \in \D} \left (   \sum_{j=1}^{\infty} \vert \phi_{f_j} (z) \vert ^2  \right) \leq 2.  $$

\end{lemma}

\begin{proof}

Since for all $z \in \D, $ $$\epsilon^2 \leq \sum_{j=1}^{\infty} \vert f_{c_j} +  \phi_{f_j} (z) \vert ^2 \leq 1,$$

we have that for each $N \in \mathbb{N},$

$$  \sum_{j=1}^{N} \vert f_{c_j} +  \phi_{f_j} (z) \vert ^2 \leq 1.$$

But, $ \{ { \phi_{f_j}} \}_{j=1}^{N}   \subset \mathbb{I}$ and $\mathbb{I}$ is a proper ideal, so by the corona theorem
$$ \inf_{z\in \D}  \sum_{j=1}^{N} \vert  \phi_{f_j} (z) \vert ^2 =0.$$

This means that for each $N$  
$$\sum_{j=1}^{N} \vert f_{c_j} \vert ^2 \leq 1, \text{ and hence} \; \; \sum_{j=1}^{\infty} \vert f_{c_j} \vert ^2 \leq 1.$$

Thus, (b) holds, since for $z\in \D$ 

$$  \left ( \sum_{j=1}^{\infty} \vert \phi_{f_j} (z) \vert ^2  \right)^{\frac{1}{2}} \leq \left ( \sum_{j=1}^{\infty} \vert f_{c_j} + \phi_{f_j}(z) \vert ^2  \right)^{\frac{1}{2}} +  \left ( \sum_{j=1}^{\infty} \vert f_{c_j}  \vert ^2  \right)^{\frac{1}{2}}  \leq 2.$$

Now by the Rosenblum- Tolokonnikov-Uchiyama version of the corona theorem, since $ \{ { \phi_{f_j}} \}_{j=1}^{\infty}   \subset \mathbb{I}$ and $\mathbb{I}$ is a proper closed ideal and \\
$ \underset {z\in \D} \sup  \sum_{j=1}^{\infty} \vert \phi_{f_j} (z) \vert ^2  \leq 2 < \infty,$ we have 

$$ \inf_{z\in \D}  \sum_{j=1}^{\infty} \vert \phi_{f_j} (z) \vert ^2 = 0.$$
Thus there exist $\{ z_k \}_{k=1}^{\infty} \subset \D$ so that $ \underset { {k \rightarrow \infty}} \lim  \sum_{j=1}^{\infty} \vert \phi_{f_j} (z_k) \vert ^2 = 0.$\\

Therefore, from 
$$  \epsilon \leq \left ( \sum_{j=1}^{\infty} \vert f_{c_j} + \phi_{f_j}(z_k) \vert ^2  \right)^{\frac{1}{2}}  \leq  \left ( \sum_{j=1}^{\infty} \vert f_{c_j}  \vert ^2  \right)^{\frac{1}{2}} +   \left ( \sum_{j=1}^{\infty} \vert \phi_{f_j}(z_k)  \vert ^2  \right)^{\frac{1}{2}},$$ 
we deduce that 
$$ \epsilon ^2 \leq  \sum_{j=1}^{\infty} \vert f_{c_j}  \vert ^2 .$$
So (a) follows.
\end{proof}

Now we are ready to prove our theorems.

\section{The Proofs}

\begin{proof}[Proof of Theorem \ref{th1}]

Let $ F \in \left( H^{\infty}_{\mathbb{I}}(\mathbb{D}) \right)_{l^2}$,  and suppose
$$ 0< \epsilon ^2 \leq  F(z)F(z)^{\ast}\;  \leq 1 \; \text{ for all } \; z\in \D. $$
 Then we know that  there is a corona solution for $F$, say  $ G$,   which lies in $ \left( H^{\infty}(\mathbb{D}) \right)_{l^2}$ such that
\begin{align*} 
F(z)G(z)^T = &1, \;  \text{ for all } \; z\in \D \;  \text{and}\\
\Vert G \Vert _{\infty} \leq  & \frac{9}{\epsilon ^2} \ln \left( \frac{1}{\epsilon ^2}\right).
\end{align*}
Our aim is finding $U \in  \left (H^{\infty}_{\mathbb{I}}(\mathbb{D})\right)_{l^2}$ such that $F(z)U(z)^T = 1 \; \text{ for all } \; z\in \D. $ For this, we construct a new solution by adding a correction term to $G(z)^T$.

Write $ F(z) = F_c + \phi_F (z) $, where $ F_c= \{f_{c_1}, f_{c_2},...\} \in l^2 $ and $ \phi_F= \{ \phi_{f_1}, \phi_{f_2},...\}  \in \mathbb{I}_{l^2} $.

 Using (\ref{relation}), we have that 

$$I_{l^2} = (F(z)G(z)^T)I = G(z)^TF(z) + Q_{F(z)}Q_{G(z)}^T $$
This implies that
\begin{equation} \label{relation2}
I_{l^2}= G(z)^T F_c  + Q_{F(z)}Q_{G(z)}^T+ G(z)^T \phi_F (z).
\end{equation}

Applying $F_c^{\star} $ to (\ref{relation2}), we get

$$F_c^{\star} = G(z)^T F_c F_c^{\star}  + Q_{F(z)}Q_{G(z)}^T F_c^{\star} + G(z)^T \phi_F (z) F_c^{\star}.$$
Also, from Lemma 2.2, we know that $ \Vert F_c \Vert ^2 > 0 $, so
$$ \frac  {F_c^{\star} } { \Vert F_c \Vert ^2} = G(z)^T +  Q_{F(z)}Q_{G(z)}^T \frac  {F_c^{\star} } { \Vert F_c \Vert ^2} +  G(z)^T \phi_F (z) \frac  {F_c^{\star} } { \Vert F_c \Vert ^2} . $$




Thus,
\begin{equation}\label{sol}
 G(z)^T+Q_{F(z)}Q_{G(z)}^T  \frac  {F_c^{\star} } { \Vert F_c \Vert ^2} =  \frac  {F_c^{\star} } { \Vert F_c \Vert ^2} -G(z)^T \phi_F (z)  \frac  {F_c^{\star} } { \Vert F_c \Vert ^2}.
\end{equation}
Define $$ U(z)^T:=G(z)^T+Q_{F(z)}Q_{G(z)}^T  \frac  {F_c^{\star} } { \Vert F_c \Vert ^2}.$$
Using  (\ref{ker}), we can clearly see that $$F(z)U(z)^T=F(z)G(z)^T+ F(z)Q_{F(z)}Q_{G(z)}^T  \frac  {F_c^{\star} } { \Vert F_c \Vert ^2}=F(z)G(z)^T=1,\;  \text{for all } \; z \in \D.$$
Also, the right side of  ( \ref{sol}) shows that the  solution $U$ is in $  \left (H^{\infty}_{\mathbb{I}}(\mathbb{D})\right)_{l^2} $. 

For the norm estimate, we have that $\Vert U \Vert _{\infty} \leq   \left( { 1+ \frac{1}{\Vert F_c \Vert}} \right) \Vert G \Vert_{\infty}.$ \\

Hence, 

 $$\Vert U \Vert _{\infty} \leq \left( { 1+ \frac{1}{\Vert F_c \Vert}}  \right) \frac{9}{\epsilon ^2} \ln \left( \frac{1}{\epsilon ^2}\right).$$ 
This completes the proof of Theorem 1. 
\end{proof}

\begin{proof}[Proof of Theorem \ref{th2}]

Let $ F \in H^{\infty}_{\mathbb{I}}(\D)_{l^2} $,  and suppose
$$ \vert h(z) \vert  \leq  F(z)F(z)^{\ast} \psi \left( F(z)F(z)^{\ast} \right) \;  \leq 1 \; \text{ for all } \; z\in \D $$
 By Treil's theorem,  there exists $ G \in H_{l^2}^{\infty}(\D) $ such that

$$F(z)G(z)^T = h(z) \;\;  \text{ for all } \; z\in \D $$  
and  $\Vert G \Vert _{\infty} \leq  C_0, $  where $C_0$ is the estimate for the $ H^{\infty}(\D)$-solution obtained in \cite {T3}.\\

Writing $ F(z) = F_c + \phi_F (z),  h(z)=h_c +\phi_h(z) $ and using the relation  (\ref{relation}) as in the proof of Theorem \ref{th1}, we get





 \begin{equation} \label{relation3}
h_c \frac{F_c^\ast}{ \Vert F_c \Vert^2} + \left( \phi_h -  G(z)^T \phi_F (z) \right) \frac{F_c^\ast}{ \Vert F_c \Vert^2} = G(z)^T   + Q_{F(z)}Q_{G(z)}^T \frac{F_c^ \ast }{\Vert F_{c}\Vert ^2} .
\end{equation}

Define $$V(z)^T:=G(z)^T+Q_{F(z)}Q_{G(z)}^T \frac{F_c^ \ast }{\Vert F_{c}\Vert ^2}$$
It's clear that $$F(z)V(z)^T=h(z),\;  \text{for all } \; z \in \D.$$
Since $G \in \left( H^{\infty}(\D) \right)_{l^2}$  and the elements of $\phi_F $  are in $\mathbb{I}$, the left side of the equation  (\ref{relation3}) shows that the  solution $V$ is in $ (H^{\infty}_{\mathbb{I}}(\D))_{l^2} $. 

As in the corona theorem, for the norm estimate, we have that $\Vert V \Vert _{\infty} \leq  \left( { 1+ \frac{1}{\Vert F_c \Vert}} \right) \Vert G \Vert_{\infty} \leq  C_0 \left( { 1+ \frac{1}{\Vert F_c \Vert}} \right),$  where $C_0$ is the norm of the $H^{\infty}(\D)$ solution, $G$, obtained in \cite{T3}.



\end{proof}
\begin{proof}[Proof of Corollary \ref{cr}] 
The proof of this corollary follows similarly as the proof of Theorem \ref{th2}  by using Wolff's Theorem  instead of Treil's Theorem. 
\end{proof}

\textit{Acknowledgement}: The author would like to thank the reviewer for the thorough and constructive review, which improved the over-all presentation of this paper significantly. Also, the author would like to thank T. Trent for his helpful comments.


\begin{thebibliography}{99}

\bibitem{carleson} L. Carleson, \emph{Interpolation by bounded analytic functions and the corona
problem}, Annals of Math. \textbf{76} (1962), 547-559.

\bibitem{cegrell1} U. Cegrell, \emph{A generalization of the corona theorem in the unit disc}, Math. Z. \textbf{203} (1990), 255-261

\bibitem{cegrel2l} \bysame, \emph{Generalizations of the corona theorem in the unit disc}, Proc. Royal
Irish Acad. \textbf{94} (1994), 25-30.


%

\bibitem{DPRS} K. R. Davidson, V. I. Paulsen, and M. Ragupathi, and D. Singh, \emph{ A constrained
Nevanlinna-Pick theorem}, Indiana Math. J. \textbf{58}  (2009), no.2, 709--732.

\bibitem{garnett} J. B. Garnett, \emph{Bounded Analytic Functions}, Academic Press, (2007)

\bibitem{GMN} P. Gorkin, R. Mortini, and A. Nicolau,  \emph{The generalized corona theorem},  Math. Annalen 
\textbf{301}  (1995), 135-154.


\bibitem{mortini} R. Mortini \emph{ Generating sets for Ideals of finite type in  $H^{\infty}$}, Bull. Sci. Math.  \textbf{136} (2012),  687 - 708.



\bibitem{MSW}  R. Mortini, A. Sasane, and B. Wick, \emph{The corona theorem and stable rank for $\C+BH^{\infty}(\D)$}, Houston J. Math. \textbf{36} (2010), no. 1, 289-302.

\bibitem{nikolski} N. K. Nikolski, \emph{Treatise on the Shift Operator}, Springer-Verlag, New York
(1985).

\bibitem{pathi} M. Ragupathi, \emph{Nevanlinna-Pick interpolation for $\C+BH^{\infty}(\D)$}, Integral Equa.
Oper. Theory \textbf{63} (2009), 103-125.

\bibitem{pau} J. Pau, \emph{On a generalized corona problem on the unit disc}, Proc. Amer. Math. Soc. \textbf{133} (2004) no. 1, 167-174.

\bibitem{rao} K. V. R. Rao, \emph{ On a generalized corona problem}, J. Analyse Math. \textbf{18} (1967), 277-278.

\bibitem{von}  M. V. Renteln, \emph{Finitely generated  ideals in the Banach algebra $H^{\infty}$ }, Collectanea Mathematica  \textbf{26} (1975), 3-14.

\bibitem{rosenblum} M. Rosenblum, \emph{A corona theorem for countably many functions}, Integral
Equa. Oper. Theory \textbf{3} (1980), no. 1, 125-137.

\bibitem{ryle1} J. Ryle, \emph{A corona theorem for certain subalgebras of $
H^{\infty}(\D) $}, Dissertation, The University of Alabama, (2009).

\bibitem{RT} J. Ryle and T. Trent, \emph{A corona theorem for certain subalgebras of
$H^{\infty}(\D)$}, Houston J. Math \textbf{37} (2011), no. 4, 1211-1226.
\bibitem{scheinberg} S. Scheinberg, \emph{Cluster sets and corona theorems in Banach spaces of
analytic functions}, Lecture Notes in Mathematics, Springer, New York, 1976

\bibitem{ryle2} J Ryle and T. Trent, \emph{A corona theorem for certain subalgebras of
$H^{\infty}(\D)$ II}, Houston J. Math \textbf{38} (2012), no. 4, 1277-1295.
\bibitem{scheinberg} S. Scheinberg, \emph{Cluster sets and corona theorems in Banach spaces of
analytic functions}, Lecture Notes in Mathematics, Springer, New York, 1976.


\bibitem{tolokonnikov} V. A. Tolokonnikov, \emph{The corona theorem in algebras of smooth
functions}, Translations (American Mathematical Society), \textbf{149} (1991) no. 2, 61-95.

\bibitem{treil2} S. R. Treil, \emph{Estimates in the corona theorem and ideals of $ H^{\infty} $: A
problem of T. Wolff}, J. Anal. Math \textbf{87} (2002), 481-495

\bibitem{T3}\bysame, \emph{The problem of ideals of $H^{\infty}(\mathbb{D})$:
Beyond the exponent $\frac{3}{2}$,} J. Fun. Anal. \textbf{253} (2007), 220-240.


\bibitem{estimate} T. Trent, \emph{An estimate for ideals in $ H^{\infty}(\D) $}, Integr. Equat.
Oper. Th. \textbf{53} (2005), 573-587.

\bibitem{Q} \bysame, \emph{An $H^2$ corona theorem on the bidisk for infinitely many functions}, Linear Alg. and App. \textbf{379} (2004), 213-227.

\bibitem{anote} \bysame, \emph{A note on multiplication algebras on reproducing kernel Hilbert
spaces}, Proc. Amer. Math. Soc. \textbf{136} (2008), 2835-2838.

\bibitem{u} A. Uchiyama, \emph{Corona theorems for countably many functions and estimates for their
solutions}, preprint, UCLA, 1980.


\bibitem{wolff} T. Wolff, \emph{A refinement of the corona theorem}, in Linear and Complex Analysis
Problem Book, by V. P. Havin, S. V. Hruscev, and N. K. Nikolski (eds.), Springer-Verlag, Berlin
(1984).




\end{thebibliography}
\end{document}